\documentclass[proc-l]{amsart}
\usepackage{amssymb,amscd,amsfonts,amsbsy,nicefrac,mathtools}
\usepackage[utf8]{inputenc}
\usepackage{amsmath,amsthm,amssymb,amsfonts,enumitem}
\usepackage{mathrsfs}
\usepackage[nobysame]{amsrefs}
\usepackage{hyperref}
\usepackage[pdftex]{graphicx,color}
\usepackage{graphicx}
\usepackage{marvosym}
\usepackage{scalerel}
\usepackage{bbm}
\usepackage{bm}
\usepackage{setspace}
\usepackage{tikz}
\usepackage[foot]{amsaddr}
\usepackage[normalem]{ulem}

\usepackage{fancyhdr} 
\usepackage{ulem} 

\usepackage{xcolor}

\newtheorem{thm}{Theorem}[section]
\newtheorem{lem}[thm]{Lemma}
\newtheorem{lemma}[thm]{Lemma}
\newtheorem{cor}[thm]{Corollary}
\newtheorem{prop}[thm]{Proposition}
\newtheorem{rem}[thm]{Remark}

\newcommand{\C}{\mathbb C}

\newcommand{\Z}{\mathbb{Z}}
\newcommand{\N}{\mathbb{N}}

\newcommand{\D}{\mathbb{D}}

\newcommand{\dd}{\,\mathrm{d} }

\newcommand{\abs}[1]{\left|#1\right|}
\newcommand{\set}[1]{\left\{#1\right\}}
\newcommand{\brkt}[1]{\left(#1\right)}
\newcommand{\jap}[1]{\big\langle #1\big\rangle}

\newcommand{\m}[1]{\begin{equation*}
\begin{split}
#1
\end{split}
\end{equation*}}
\newcommand{\nm}[2]{\begin{equation}\label{#1}
\begin{split}
#2
\end{split}
\end{equation}}

\begin{document}


\title{{Szeg\H{o} Recurrence for Multiple Orthogonal Polynomials on the Unit Circle}}
\date{\today}
\author{Marcus Vaktnäs$^{1,*}$}
\email{marcus.vaktnas@math.uu.se}
\author{Rostyslav Kozhan$^1$}
\email{rostyslav.kozhan@math.uu.se}
\address{$^{1}$Department of Mathematics, Uppsala University, S-751 06 Uppsala, Sweden}
\begin{abstract}
    We investigate polynomials that satisfy simultaneous orthogonality conditions with respect to several measures on the unit circle. We generalize the direct and inverse Szeg\H{o} recurrence relations, identify the analogues of the Verblunsky coefficients, and prove the Christoffel--Darboux formula. These results stand directly in analogue with the nearest neighbour recurrence relations from the real line counterpart.
\end{abstract}
\maketitle

\section{Introduction}


The theory of multiple orthogonal polynomials on the real line (MOPRL for short) deals with polynomials that are orthogonal with respect to several measures on the real line. These polynomials initially appeared in Hermite--Padé approximations of Markov functions. Today MOPRL theory is very well developed with applications in approximation theory, spectral theory, random matrix theory, and integrable probability. See \cite{Applications} for a recent quick introduction to MORPL and its applications, and \cite{Aptekarev,Ismail,bookNS} for a more thorough treatment.


While the MOPRL theory is well-developed, the theory of multiple orthogonal polynomials {\it on the unit circle} (MOPUC) is still at its infancy. It was introduced in \cite{mopuc}, motivated by applications in approximation theory and prediction theory. In particular, these polynomials appear when studying Hermite--Padé approximations of Carathéodory functions. Since then, MOPUC has only been further studied once, in \cite{mopuc recurrence}. The goal of this article is to encourage further development in MOPUC, by deriving 
the analogues of two MOPRL results that are important milestones in the theory.

The first result is the Christoffel--Darboux formula
by Daems and Kuijlaars~\cite{DaeKui}  (see also~\cite{AFFM} for the more general setting), which is the starting point of many further applications of MOPRL to random matrix theory and Markov processes for non-colliding particles, see~\cite{Bor,DaeKui07,Kui,KuiMar,Dui18,DFK}.

The other important advance in this theory was the paper \cite{NNR} by Van Assche  showing that MOPRL satisfy the so-called nearest neighbour recurrence relations, which is the generalization of the three-term recurrence relation of the usual orthogonal polynomials on the real line. In particular, these relations 
became a simple and natural tool for studying asymptotics of MOPRL along {\it every} direction of $\Z_+^r$, rather than just along the stepline multi-indices (see \cite{AK,NeuvA,NdavA,vA16,MFS,SwivA} among others). The recurrences also provide the connection of MOPRL to the spectral theory of Jacobi operators on trees, a very recent important development \cite{Jacobi operator on trees,ADY2,DY} in this area. 


We believe that the lack of progress in MOPUC is a consequence of not having the correct analogue of the Szeg\H{o} recurrence relation from the theory of orthogonal polynomials on the unit circle (OPUC). In this paper we present recurrence relations that not only generalize the recurrence coefficients of OPUC, but are also a perfect analogue of the nearest neighbour recurrence relation from MOPRL. This opens up the theory to further progress, which we illustrate by proving a Christoffel--Darboux formula.




We start by reviewing the basics of {OPUC.} Let  $\mu$ be a probability measure on the unit circle $\partial\D = \set{z\in\C:|z|=1}$ having infinite support. The associated inner product on $L^2(\mu)$ is
\begin{equation}\label{eq:inner}
    \jap{f(z),g(z)} = \int_0^{2\pi} f(e^{i\theta}) \overline{g(e^{i\theta})} \dd\mu(e^{i\theta}).
\end{equation}
Applying the Gram--Schmidt algorithm to the sequence of monomials $\{z^j\}_{j=0}^\infty$ in $L^2(\mu)$, one obtains the sequence of monic orthogonal polynomials $\{\Phi_n(z)\}_{n=0}^\infty$ satisfying
\begin{equation}
    \jap{\Phi_n(z),z^j} = 0, \qquad j = 0,\ldots,n-1.
\end{equation}

The fundamental result of OPUC
is the Szeg\H{o} recurrence relations given by
\begin{align}
\label{eq:scalar szego 1}
\Phi_{n+1}(z) & = z\Phi_n(z) + {\alpha}_{n+1}\Phi^*_{n}(z),
\\
\label{eq:scalar szego 2}
\Phi^*_{n+1}(z) & =  \Phi^*_n(z) + \bar\alpha_{n+1}z\Phi_{n}(z),
\end{align}
for $n \in \Z_+:=\{0,1,2,\ldots\}$, where $\Phi_n^*(z) = z^n\overline{\Phi_n(1/\bar z)}$ is the reversed polynomial. Equivalently, $\Phi_n^*$ can  be defined as the unique polynomial that satisfies 
\begin{equation}\label{eq:Phi^*}
    \jap{\Phi^*_n(z),z^p} = 0, \qquad p=1,\ldots,n,
\end{equation}
with the normalization $\Phi_n^*(0)=1$. We also have the inverse Szeg\H{o} recurrence
\begin{align}
    \label{eq:scalar szego 3}
    \Phi_{n+1}(z) &= {\alpha}_{n+1}\Phi^*_{n+1}(z) + \rho_{n+1}z\Phi_n(z),
    \\
    \label{eq:scalar szego 4}
    \Phi^*_{n+1}(z) &= \bar\alpha_{n+1}\Phi_{n+1}(z) + \rho_{n+1}\Phi^*_n(z),
\end{align}
where $\rho_n = 1 - \abs{\alpha_n}^2$. 

The recurrence coefficients $\alpha_n$ (with $n\in\N:=\{1,2,3,\ldots\}$) belong to complex unit disc $\D = \set{z \in \C : \abs{z} < 1}$, and are called the Verblunsky coefficients of $\mu$ (also sometimes referred to as the Schur, Geronimus, or reflection coefficients). Note that we are using $\alpha_{n+1}$ in place where it is traditional (nowadays) to use $-\bar\alpha_n$ (see the discussion in~\cite[p.10]{Simon}). With such a choice we get $\Phi_{n}(0) = \alpha_n$ and $\Phi_n^*(z) =\bar \alpha_n z^n + o(z^{n})$, which will be natural for our purposes later. Note also that our $\rho_n$ is $1 - \abs{\alpha_n}^2$ instead of the traditional $(1 - \abs{\alpha_n}^2)^{1/2}$. 



Multiple orthogonal polynomials on the unit circle are polynomials that satisfy orthogonality conditions with respect to a system of measures $\bm{\mu} = \set{\mu_j}_{j=1}^r$. For a multi-index $\bm{n} = (n_1,\dots,n_r)\in\Z_+^r$, we write $\abs{\bm{n}} = n_1+\ldots+n_r$, and the polynomials $\Phi_{\bm{n}}$ we want to consider are monic with $\deg{\Phi_{\bm{n}}} = \abs{\bm{n}}$ and 
\nm{eq:mopuc eq0}{\jap{\Phi_{\bm{n}}(z),z^p}_j = 0,\qquad p = 0,\dots,n_j-1 ,\qquad j = 1,\dots,r,}
where $\jap{\cdot,\cdot}_j$ is the inner product~\eqref{eq:inner} but with $\mu_j$ instead of $\mu$. Such polynomials are called the type II multiple orthogonal polynomials. In analogy with~\eqref{eq:Phi^*}, we also want to consider polynomials $\Phi_{\bm{n}}^*$ with $\deg\Phi_{\bm{n}}^* \leq \abs{\bm{n}}$,
$\Phi_{\bm{n}}^*(0)=1$ and 
\nm{eq:mopuc eq0 2}{\jap{\Phi_{\bm{n}}^*(z),z^p}_j = 0,\qquad p = 1,\ldots,n_j ,\qquad j = 1,\dots,r.}
We stress that these are no longer the reversed polynomials (unless $r=1$).

One of the main results of this paper is that $\Phi_{\bm{n}}$ and $\Phi^*_{\bm{n}}$  satisfy the following Szeg\H{o}-type recurrences: 
\begin{align}
\label{eq:second recurrence0}
\Phi_{\bm{n}}(z) & = \alpha_{\bm{n}}\Phi_{\bm{n}}^*(z) + \sum_{i = 1}^r \rho_{\bm{n},i}z\Phi_{\bm{n}-\bm{e}_i}(z)
\\
\label{eq:first recurrence0}
\Phi_{\bm{n}}^*(z) & = \Phi_{\bm{n}-\bm{e}_k}^*(z) + \beta_{\bm{n}}z\Phi_{\bm{n}-\bm{e}_k}(z),
\end{align}
for some coefficients $\alpha_{\bm{n}}$, $\beta_{\bm{n}}$, and $\rho_{\bm{n},1},\dots,\rho_{\bm{n},r}$. These generalize the recurrences~\eqref{eq:scalar szego 3} and~\eqref{eq:scalar szego 2}.
Note that if all the $\alpha_{\bm{n}}$,  $\beta_{\bm{n}}$,  $\rho_{\bm{n},j}$ are known, then equations~\eqref{eq:second recurrence0}  and~\eqref{eq:first recurrence0}
are sufficient to compute all $\Phi_{\bm{n}}$'s and $\Phi_{\bm{n}}^*$'s recursively. 

The generalizations of \eqref{eq:scalar szego 4} and \eqref{eq:scalar szego 1} appear as easy consequences of \eqref{eq:second recurrence0} and \eqref{eq:first recurrence0}, however these equations seem less natural and the generalization of \eqref{eq:scalar szego 1} gives no information in the case $\beta_{\bm{n}} = 0$. Interestingly, the situation is opposite for type I multiple orthogonal polynomials
: we get generalizations of \eqref{eq:scalar szego 1} and \eqref{eq:scalar szego 4}, while the generalizations of \eqref{eq:scalar szego 2} and \eqref{eq:scalar szego 3} appear less natural, and the generalization of \eqref{eq:scalar szego 3} vanishes in the case $\alpha_{\bm{n}} = 0$. 

The coefficients $\{\alpha_{\bm{n}}\}$ and $\{\beta_{\bm{n}}\}$ should be viewed as the generalized Verblunsky/reflection coefficients.  Indeed, $\Phi_{\bm{n}}(0) = \beta_{\bm{n}}$ and $\Phi_{\bm{n}}^*(z) = \alpha_{\bm{n}} z^{|{\bm{n}}|} + o(z^{|{\bm{n}}|})$, similarly to the usual OPUC. Furthermore, for the marginal indices $\bm{n} = j \bm{e}_k$, the recurrences \eqref{eq:first recurrence0} and \eqref{eq:second recurrence0} become the usual Szeg\H{o} recurrences \eqref{eq:scalar szego 2} and ~\eqref{eq:scalar szego 3}, with $\alpha_{j \bm{e}_k}$, $\beta_{j \bm{e}_k}$, $\rho_{j \bm{e}_k,k}$ reduced to the usual OPUC recurrence coefficients $\alpha_{j}(\mu_k)$,   $\bar\alpha_{j}(\mu_k)$,  $\rho_{j}(\mu_k)$ associated with $\mu_k$, respectively (or  $-\bar\alpha_{j-1}(\mu_k)$,   $-\alpha_{j-1}(\mu_k)$,  $\rho_{j-1}(\mu_k)^2$ in the notation of ~\cite{Simon}), while  $\rho_{j \bm{e}_k,m} = 0$ for $m\ne k$. We note that the importance of 
the multiple Verblunsky coefficient $\alpha_{\bm{n}}$ was also observed in ~\cite{mopuc recurrence} (where it was denoted by $\delta_{n,m}$).

It is worth mentioning that for the recurrence relations \eqref{eq:second recurrence0} and \eqref{eq:first recurrence0} to hold, it is necessary that the polynomials appearing in these equations are uniquely determined by~\eqref{eq:mopuc eq0} and~\eqref{eq:mopuc eq0 2} (we say that the corresponding indices are then {\it normal}). This is automatic in the one-measure case, but for multiple orthogonality this is not always true, both for MOPRL and MOPUC. In the MOPRL setting, there are several wide classes of systems (Angelesco, AT, Nikishin) where normality can be shown to hold at every index. 
It would be very valuable to find analogous MOPUC classes since explicit examples of MOPUC systems where normality is proven are currently rare  (see \cite[Sect 3]{mopuc recurrence} and \cite[Sect 4]{mopuc}). Note however that normality of an index $\bm{n}$  can be stated in terms of a certain determinantal condition 
on the moments of $\set{\mu_j}_{j=1}^r$. Systems of measures that satisfy such a  condition therefore form a codimension one submanifold. From this point of view relations \eqref{eq:second recurrence0} and \eqref{eq:first recurrence0} should be regarded to hold generically.

Both the results and the methods in this paper were heavily inspired by  Van Assche's MOPRL paper \cite{NNR} {(see also \cite{Ismail})}. In order to understand the current paper however, it is not necessary to know any background from MOPRL, as we tried to make the paper  self-contained and accessible to a broad audience.

The structure of the paper is as follows. In Section~\ref{ss:normality} we discuss basic definitions and the question of uniqueness of multiple orthogonal polynomials. In Section~\ref{ss:typeII} we prove Szeg\H{o}'s recurrences for type II and II$^*$ polynomials, and in Section~\ref{ss:typeI} we do the same for type I and I$^*$. In Section~\ref{ss:compatibility} we show that the recurrence coefficients satisfy a set of partial difference equations very similar to the real line counterpart. Finally, in Section~\ref{ss:CD} we prove the Christoffel--Darboux formula.




\section{Normality}\label{ss:normality}




Given a multi-index $\bm{n}\in\Z_+^r$, a type II multiple orthogonal polynomial is a non-zero polynomial $\Phi_{\bm{n}}$ such that $\deg{\Phi_{\bm{n}}} \leq \abs{\bm{n}}$ and
\nm{eq:mopuc eq}{\jap{\Phi_{\bm{n}},z^p}_j = 0,\qquad p = 0,\dots,n_j-1 ,\qquad j = 1,\dots,r.}
We also define a type II$^*$ multiple orthogonal polynomial to be a non-zero polynomial $\Phi_{\bm{n}}^*$ such that $\deg\Phi_{\bm{n}}^* \leq \abs{\bm{n}}$ and
\nm{eq:mopuc eq 2}{\jap{\Phi_{\bm{n}}^*,z^p}_j = 0,\qquad p = 1,\dots,n_j ,\qquad j = 1,\dots,r.}


A type I multiple orthogonal polynomial, for the multi-index $\bm{n}$, is a non-zero vector of polynomials $\bm{\Lambda_{\bm{n}}} = (\Lambda_{\bm{n},1},\dots,\Lambda_{\bm{n},r})$ such that $\deg{\Lambda_{\bm{n},j}} \leq n_j-1$ and
\nm{eq:mopuc eq type 1}{\sum_{j = 1}^r \jap{\Lambda_{\bm{n},j}, z^p}_j = 0,\qquad p = 0,1,\dots,\abs{\bm{n}}-2.}
Lastly, we define a type I$^*$ multiple orthogonal polynomial, to be a non-zero vector of polynomials $\bm{\Lambda}_{\bm{n}}^* = (\Lambda^*_{\bm{n},1},\dots,\Lambda^*_{\bm{n},r})$ such that $\deg{\Lambda^*_{\bm{n},j}} \leq n_j-1$ and
\nm{eq:mopuc eq 2 type 1}{\sum_{j = 1}^r \jap{\Lambda^*_{\bm{n},j}, z^p}_j = 0,\qquad p = 1,2,\dots,\abs{\bm{n}}-1.}
We may also refer to each $\Lambda_{\bm{n},j}$ and $\Lambda_{\bm{n},j}^*$ as type I and type I$^*$ polynomials, respectively. Note that $\Lambda_{\bm{n},j} = \Lambda_{\bm{n},j}^* = 0$ is the only possibility when $n_j = 0$ (we take the degree of $0$ to be $-\infty$). For $\bm{n} = \bm{0}$ we just define $\bm\Lambda_{\bm{0}} = \bm{0}$ and $\bm\Lambda_{\bm{0}}^* = \bm{0}$ as the only type I and type I$^*$ polynomials.




\textup{Consider the matrix 
\nm{eq:mopuc matrix}
{ M_{\bm{n}} = \begin{pmatrix}
\nu_1^0 & \nu_1^1 & \cdots & \nu_1^{\abs{\bm{n}}-1} \\
\nu_1^{-1} & \nu_1^0 & \cdots & \nu_1^{\abs{\bm{n}}-2} \\
\vdots & \vdots & \ddots & \vdots \\
\nu_1^{1-n_1} & \nu_1^{2-n_1} & \cdots & \nu_1^{\abs{\bm{n}}-n_1} \\
\hline
& & \vdots \\
\hline
\nu_{r}^0 & \nu_{r}^1 & \cdots & \nu_{r}^{\abs{\bm{n}}-1} \\
\nu_{r}^{-1} & \nu_{r}^0 & \cdots &  \nu_{r}^{\abs{\bm{n}}-2} \\
\vdots & \vdots & \ddots & \vdots \\
\nu_{r}^{1-n_r} & \nu_{r}^{2-n_r} & \cdots & \nu_{r}^{\abs{\bm{n}}-n_r} \\
\end{pmatrix},}
where  $\nu_j^p = \int z^p \dd\mu_j(z)$ are the moments of $\mu_j$.
    We say that the index $\bm{n}\ne\bm{0}$ is normal if $\det M_{\bm{n}}\ne 0$.
    }
 This condition ensures uniqueness of our polynomials at the location $\bm{n}$ if we choose the appropriate normalization. We take $\bm{n} = \bm{0}$ to always be normal.

\begin{lem} \label{lem:normality}
An index $\bm{n}\neq\bm{0}$ is normal if and only if any of the following conditions hold:
\begin{itemize}
    \item[(i)] $\deg \Phi_{\bm{n}} = \abs{\bm{n}}$ for every type II polynomial $\Phi_{\bm{n}}$.
    \item[(ii)] $\Phi_{\bm{n}}^*(0) \neq 0$ for every type II$^*$ polynomial $\Phi_{\bm{n}}^*$.    
    \item[(iii)] $\sum_{j = 1}^r \jap{\Lambda_{\bm{n},j}, z^{\abs{\bm{n}}-1}}_j \neq 0$ for every non-zero type I polynomial $\bm{\Lambda_{\bm{n}}}$.
    \item[(iv)] $\sum_{j = 1}^r \jap{\Lambda^*_{\bm{n},j}, 1}_j \neq 0$ for every non-zero type I$^*$ polynomial $\bm{\Lambda_{\bm{n}}}^*$.
\end{itemize}
\end{lem}

\begin{rem}\label{rem:a-d}
    \textup{We show in the proof that  normality of $\bm{n}\ne \bm{0}$ 
    is also equivalent to any of the following statements:
    \begin{itemize}
    \item[(a)] There is a unique monic type II multiple orthogonal polynomial $\Phi_{\bm{n}}$ such that $\deg{\Phi_{\bm{n}}} = \abs{\bm{n}}$.
    \item[(b)] There is a unique type II$^*$ multiple orthogonal polynomial $\Phi^*_{\bm{n}}$ such that $\Phi^*_{\bm{n}}(0) = 1$.
    \item[(c)] There is a unique type I multiple orthogonal polynomial $(\Lambda_{\bm{n},1},\dots,\Lambda_{\bm{n},r})$ such that $\sum_{j = 1}^r \jap{\Lambda_{\bm{n},j}, z^{\abs{\bm{n}}-1}}_j = 1$. 
    \item[(d)] There is a unique type I$^*$ multiple orthogonal polynomial $(\Lambda_{\bm{n},1}^*,\dots,\Lambda_{\bm{n},r}^*)$ such that $\sum_{j = 1}^r \jap{\Lambda^*_{\bm{n},j}, 1}_j = 1$.
    \end{itemize}
    These normalizations will be used in all the future sections. Note that $\set{\Phi_{j\bm{e}_k}}_{j = 0}^{\infty}$ are the monic orthogonal polynomials with respect to $\mu_k$, and $\Lambda_{(j+1)\bm{e}_k,k} = \tfrac{1}{||\Phi_{j\bm{e}_k}||^2_k} \Phi_{j\bm{e}_k}$.}
\end{rem}


\begin{proof}
Solving the system (\ref{eq:mopuc eq}) for polynomials of the form $c_{\abs{\bm{n}}-1}z^{\abs{\bm{n}}-1} + \dots + c_{0}$ results in a linear system with coefficient matrix $M_{\bm{n}}$. Hence $c_{\abs{\bm{n}}-1} = \dots = c_{0} = 0$ if and only if this matrix is invertible. If we instead solve for $z^{\abs{\bm{n}}} + c_{\abs{\bm{n}}-1}z^{\abs{\bm{n}}-1} + \dots + c_{0}$ we again get a linear system with coefficient matrix $M_{\bm{n}}$. This proves that normality is equivalent to (i) and (a). The system (\ref{eq:mopuc eq 2}) for $c_{\abs{\bm{n}}}z^{\abs{\bm{n}}} + \dots + c_1z$, as well as $c_{\abs{\bm{n}}}z^{\abs{\bm{n}}} + \dots + c_1z + 1$, also has the same coefficient matrix. This proves that normality is equivalent to (ii) and (b). 
Next, the system of equations
\m{\sum_{j = 1}^r \jap{\Lambda_{\bm{n},j}, z^p}_j = 0, \qquad p = 0,1,\dots,\abs{\bm{n}}-1}
is homogeneous with coefficient matrix $M_{\bm{n}}^T$, so the existence of $(\Lambda_{\bm{n},1},\dots,\Lambda_{\bm{n},r}) \neq \bm{0}$ satisfying this extended set of equations is equivalent to $M_{\bm{n}}^T$ not being invertible. Similarly, if we change the right hand side to be $1$ for $p = \abs{\bm{n}}-1$ we see that we have a unique solution exactly when $M_{\bm{n}}^T$ is invertible. The same argument works for $(\Lambda^*_{\bm{n},1},\dots,\Lambda^*_{\bm{n},r})$.
\end{proof}

We will also need the following lemma, cf.~\cite[Cor. 23.1.1--23.1.2]{Ismail} for MOPRL.



\begin{lem} \label{lem:normality lemma}
$\bm{n} + \bm{e}_k$ is normal if and only if we have $\jap{\Phi_{\bm{n}},z^{n_k}}_k \neq 0$ for every type II multiple orthogonal polynomial $\Phi_{\bm{n}}$ for the index $\bm{n}$. Similarly, $\bm{n} - \bm{e}_k$ is normal if and only if $\deg{\Lambda_{\bm{n},k}} = n_k - 1$ for every type I multiple orthogonal polynomial $\Lambda_{\bm{n},k}$ for the index $\bm{n}$.
\end{lem}

\begin{proof}
If $\bm{n}+\bm{e}_k$ is not normal then for this index there is some non-zero solution $\Phi_{\bm{n}+\bm{e}_k}$ of (\ref{eq:mopuc eq}) with $\deg{\Phi_{\bm{n}+\bm{e}_k}} \leq \abs{\bm{n}}$. But then $\Phi_{\bm{n}} = \Phi_{\bm{n} + \bm{e}_k}$ is a solution of (\ref{eq:mopuc eq}) for the index $\bm{n}$ as well, with $\jap{\Phi_{\bm{n}},z^{n_k}}_k = 0$. Conversely, if $\jap{\Phi_{\bm{n}},z^{n_k}}_k = 0$ and $\Phi_{\bm{n}}$ is a solution of (\ref{eq:mopuc eq}) for the index $\bm{n}$, then $\Phi_{\bm{n}}$ is a solution of (\ref{eq:mopuc eq}) for the index $\bm{n}+\bm{e}_k$. Since $\deg{\Phi_{\bm{n}}} < \abs{\bm{n}+\bm{e}_k} = \abs{\bm{n}} + 1$ it follows that $\bm{n}+\bm{e}_k$ cannot be normal. 

For the second statement, assume $n_k > 0$ (and otherwise the statement is obviously true). If $\bm{n} - \bm{e}_k$ is not normal then there is a $(\Lambda_{\bm{n}-\bm{e}_k,1},\dots,\Lambda_{\bm{n}-\bm{e}_k,r}) \neq \bm{0}$ such that
\m{\sum_{j = 1}^r \jap{\Lambda_{\bm{n}-\bm{e}_k,j}, z^p}_j = 0,\qquad p = 0,1,\dots,\abs{\bm{n}}-2.}
This is also a solution for the index $\bm{n}$, so we have a solution $(\Lambda_{\bm{n},1},\dots,\Lambda_{\bm{n},r}) \neq \bm{0}$ with $\deg{\Lambda_{\bm{n},j}} < n_j - 1$. Conversely, if $\deg{\Lambda_{\bm{n},j}} < n_j - 1$ and $(\Lambda_{\bm{n},1},\dots,\Lambda_{\bm{n},r}) \neq \bm{0}$ then this is a non-zero solution to (\ref{eq:mopuc eq type 1}) for $\bm{n} - \bm{e}_k$, and
\m{\sum_{j = 1}^r \jap{\Lambda_{\bm{n},j}, z^{\abs{\bm{n}}-2}}_j = 0,}
so $\bm{n} - \bm{e}_k$ cannot be normal.
\end{proof}

\section{Recurrence Relations for Type II Polynomials}\label{ss:typeII}

From now on, if an index $\bm{n}$ is normal then we only work with monic type II polynomials $\Phi_{\bm{n}}$, and type II* polynomials $\Phi^*_{\bm{n}}$ with $\Phi^*_{\bm{n}}(0) = 1$. These polynomials are then unique, as discussed in Remark~\ref{rem:a-d} (a) and (b). Also, if $\bm{n}\notin \Z_+^r$ we always take $\Phi_{\bm{n}}=0$. 
The next theorem is the generalization of the Szeg\H{o} recurrence relations~\eqref{eq:scalar szego 3} and ~\eqref{eq:scalar szego 2} from the OPUC theory.

\begin{thm}\label{thm:type 2 szego}
For each of the equations~\eqref{eq:second recurrence} and~\eqref{eq:first recurrence} we assume that all the $\Z_+^r$ multi-indices that appear in the corresponding equation are normal.
\begin{itemize}
    \item[(i)] 
There are complex numbers $\alpha_{\bm{n}}$ and $\rho_{\bm{n},1},\dots,\rho_{\bm{n},r}$ such that
\nm{eq:second recurrence}{
\Phi_{\bm{n}} = \alpha_{\bm{n}}\Phi_{\bm{n}}^* + \sum_{j = 1}^r \rho_{\bm{n},j}z\Phi_{\bm{n}-\bm{e}_j}.
}
    \item[(ii)] If $n_k>0$, there is a complex number $\beta_{\bm{n}}$, independent of $k$, such that 
    \nm{eq:first recurrence}{\Phi_{\bm{n}}^* = \Phi_{\bm{n}-\bm{e}_k}^* + \beta_{\bm{n}}z\Phi_{\bm{n}-\bm{e}_k}
    }
\end{itemize}
\end{thm}
\begin{rem}
    \textup{The recurrence coefficients $\alpha_{\bm{n}}$ and $\beta_{\bm{n}}$ will be referred to as the multiple Verblunsky coefficients of  the system $\bm{\mu}$. Note that the recurrence relations imply that $\alpha_{\bm{n}} = \Phi_{\bm{n}}(0)$ and $\beta_{\bm{n}}$ is the $z^{|\bm{n}|}$-coefficient of $\Phi^*_{\bm{n}}$. Even if the recurrence relations do not hold, we can still define $\alpha_{\bm{n}}$ and $\beta_{\bm{n}}$ in this way, as long as $\bm{n}$ is normal. Also note that the recurrence relations uniquely determine $\alpha_{\bm{n}}$ and $\beta_{\bm{n}}$. The same holds for $\rho_{\bm{n},k}$, assuming $n_k > 0$. If $n_k = 0$ we can adopt the convention $\rho_{\bm{n},k} = 0$. Furthermore, for indices $\bm{n}$ of the form $n_k\bm{e}_k$ we get $\beta_{n_k\bm{e}_k} = \bar\alpha_{n_k\bm{e}_k}$, which become the usual Verblunsky (reflection) coefficients of $\mu_k$, $\rho_{n_k\bm{e}_k, j} = 1-|\alpha_{n_k\bm{e}_k}|^2$, and $\rho_{n_k\bm{e}_k,m}=0$ when $m\ne j$, so that ~\eqref{eq:second recurrence} and~\eqref{eq:first recurrence} reduce to the usual Szeg\H{o} recurrence relations \eqref{eq:scalar szego 3} and ~\eqref{eq:scalar szego 2} for the measure $\mu_j$. }
\end{rem}

\begin{proof}
(i) We choose $\alpha_{\bm{n}}$ such that $(\Phi_{\bm{n}} - \alpha_{\bm{n}}\Phi_{\bm{n}}^*)(0) = 0$ and then we get
\nm{eq:recurrence derivation}{\jap{z^{-1}(\Phi_{\bm{n}} - \alpha_{\bm{n}}\Phi_{\bm{n}}^*),z^p}_j & = \jap{\Phi_{\bm{n}},z^{p+1}}_j - \alpha_{\bm{n}}\jap{\Phi_{\bm{n}}^*,z^{p+1}}_j \\ & = 0,\hspace{0.3cm} p = 0,\dots,n_i-2,\hspace{0.3cm} j = 1,\dots,r.}
Note that the polynomial $z^{-1}(\Phi_{\bm{n}} - \alpha_{\bm{n}}\Phi_{\bm{n}}^*)$ has degree at most $\abs{\bm{n}}-1$. 
Denote $S_{\bm{n}}=\{j\in\{1,\ldots,r\}:n_j>0\}$, and let   $s_{\bm{n}}$ be the cardinality of $S_{\bm{n}}$.
Solving ~\eqref{eq:recurrence derivation} for a polynomial $c_{\abs{\bm{n}}-1}z^{\abs{\bm{n}}-1}+\dots+c_0$ results in a homogeneous system with coefficient matrix $M_{\bm{n}}'$, equal to the matrix $M_{\bm{n}}$ in \eqref{eq:mopuc matrix} but with $s_{\bm{n}}$ rows removed. Since $\bm{n}$ is normal the row space of $M_{\bm{n}}'$ has dimension $\abs{\bm{n}}-s_{\bm{n}}$, so the null space of $M_{\bm{n}}'$ has dimension $s_{\bm{n}}$. Each $\Phi_{\bm{n}-\bm{e}_j}$ with $j\in S_{\bm{n}}$  solves this system of equations, and we claim they are also linearly independent. To see this, consider the equation
\m{\sum_{j\in S_{\bm{n}}} t_j \Phi_{\bm{n}-\bm{e}_j} = 0}
and suppose $k\in S_{\bm{n}}$ (so that we have $\Phi_{\bm{n}-\bm{e}_k} \neq 0$). Taking the $k$-th inner product with $z^{n_k-1}$ yields
\m{\sum_{j\in S_{\bm{n}}} t_j \jap{\Phi_{\bm{n}-\bm{e}_j},z^{n_k-1}}_k = 0.}
Every term  with $j \neq k$ vanishes by the orthogonality relations \eqref{eq:mopuc eq}, so we end up with
\m{t_k \jap{\Phi_{\bm{n}-\bm{e}_k},z^{n_k-1}}_k = 0.}
The inner product is non-zero by Lemma \ref{lem:normality lemma}, so we must have $t_k = 0$. Hence linear independence follows and $\set{\Phi_{\bm{n}-\bm{e}_j}}_{j\in S_{\bm{n}}}$ is a basis for the space of polynomials of degree at most $\abs{\bm{n}}-1$ with the orthogonality relations \eqref{eq:recurrence derivation}. This then implies
\m{z^{-1}(\Phi_{\bm{n}} - \alpha_{\bm{n}}\Phi_{\bm{n}}^*) = \sum_{j = 1}^r \rho_{\bm{n},j}\Phi_{\bm{n}-\bm{e}_j}}
for some constants $\rho_{\bm{n},1},\dots,\rho_{\bm{n},r}$ (uniquely determined by $\Phi_{\bm{n}-\bm{e}_1},\dots,\Phi_{\bm{n}-\bm{e}_r}$).


\medskip

(ii)   $\Phi_{\bm{n}}^* - \Phi_{\bm{n}-\bm{e}_k}^*$ is $0$ at $z = 0$, so it divisible by $z$. Then
    \m{\jap{z^{-1}(\Phi_{\bm{n}}^* - \Phi_{\bm{n}-\bm{e}_k}^*),z^p}_k & = \jap{\Phi_{\bm{n}}^*,z^{p+1}}_k - \jap{\Phi_{\bm{n}-\bm{e}_k}^*,z^{p+1}}_k \\ & = 0,\hspace{0.3cm} p = 0,\dots,n_k-2,\hspace{0.3cm} j = 1,\dots,r.}
    Similary, if $j \neq k$ we have the orthogonality relations 
    \m{\jap{z^{-1}(\Phi_{\bm{n}}^* - \Phi_{\bm{n}-\bm{e}_k}^*),z^p}_j & = \jap{\Phi_{\bm{n}}^*,z^{p+1}}_j - \jap{\Phi_{\bm{n}-\bm{e}_k}^*,z^{p+1}}_j \\ & = 0,\hspace{0.3cm} p = 0,\dots,n_j-1,\hspace{0.3cm} j = 1,\dots,r.}
    Hence we see that $z^{-1}(\Phi_{\bm{n}}^* - \Phi_{\bm{n}-\bm{e}_k}^*)$ satisfies all the orthogonality conditions in \eqref{eq:mopuc eq} for the index $\bm{n}-\bm{e}_k$, while having degree $\abs{\bm{n}}-1$ or $0$. Hence we get 
    \m{z^{-1}(\Phi_{\bm{n}}^* - \Phi_{\bm{n}-\bm{e}_k}^*) = \beta_{\bm{n},k}\Phi_{\bm{n}-\bm{e}_k}.}
    By comparing $z^{|\bm{n}|-1}$-coefficients, we see that $\beta_{\bm{n},k}$ is the $z^{|\bm{n}|}$-coefficient of $\Phi_{\bm{n}}^*$, hence independent of $k$, and so \eqref{eq:first recurrence} follows.
\end{proof}

One can ask what the analogues of~\eqref{eq:scalar szego 1} and \eqref{eq:scalar szego 4} are. We obtain them in Corollary~\ref{cor:szTypeIIextra2} and Corollary~\ref{cor:szTypeIIextra1}, respectively.
\begin{cor}\label{cor:szTypeIIextra1}
    Assume that all the $\Z_+^r$ multi-indices that appear in the next equation are normal. Then
    \nm{eq:third recurrence}{\Phi_{\bm{n}}^* = \beta_{\bm{n}}\Phi_{\bm{n}} + \sum_{j = 1}^r \rho_{\bm{n},j}\Phi_{\bm{n}-\bm{e}_j}^*.}
\end{cor}
\begin{proof}
Multiply both sides of \eqref{eq:second recurrence} with $\beta_{\bm{n}}$ and use \eqref{eq:first recurrence} to get
\m{\beta_{\bm{n}}\Phi_{\bm{n}} = \alpha_{\bm{n}}\beta_{\bm{n}}\Phi_{\bm{n}}^* + \sum_{j = 1}^r \rho_{\bm{n},j}(\Phi_{\bm{n}}^* - \Phi_{\bm{n}-\bm{e}_j}^*) = \brkt{\alpha_{\bm{n}}\beta_{\bm{n}} + \sum_{j = 1}^r \rho_{\bm{n},j}}\Phi_{\bm{n}}^* - \sum_{j = 1}^r \rho_{\bm{n},j} \Phi_{\bm{n}-\bm{e}_j}^*,}
and by comparing the  leading coefficients in \eqref{eq:second recurrence} we see that
\begin{equation}\label{eq:sumRhos}
    \alpha_{\bm{n}}\beta_{\bm{n}} + \sum_{j = 1}^r \rho_{\bm{n},j} = 1.
\end{equation}
\end{proof}

\begin{cor}\label{cor:szTypeIIextra2}
Assume that $\bm{n}$, $\bm{n}-\bm{e}_1,\ldots, \bm{n}-\bm{e}_r$ are normal. Define $R_{\bm{n}}$ to be the $r\times r$ matrix $\left(\rho_{\bm{n},k} \right)_{j,k=1}^r$.
\begin{itemize}
    \item[(i)]  Assume that $\beta_{\bm{n}}\ne 0$. Define $A_{\bm{n}}=\tfrac{1}{\beta_{\bm{n}}} (I-R_{\bm{n}})$. Then
    \begin{align}
	\label{eq:szego1rewritten1}
	&
	\begin{pmatrix}
		\Phi_{\bm{n}}-z \Phi_{\bm{n}-\bm{e}_1} 
		\\
		\vdots
		\\
		\Phi_{\bm{n}}-z \Phi_{\bm{n}-\bm{e}_r} 
	\end{pmatrix}
	= 
	A_{\bm{n}} 
	\begin{pmatrix}
		\Phi^*_{\bm{n}-\bm{e}_1} 
		\\
		\vdots
		\\
		 \Phi^*_{\bm{n}-\bm{e}_r} 
	\end{pmatrix}. 
\end{align}
\item[(ii)] Assume that $\alpha_{\bm{n}}\ne 0$. Define $A_{\bm{n}}^{-1}=\tfrac{1}{\alpha_{\bm{n}}} \left((1-\sum_{l=1}^r \rho_{\bm{n},l})I+R_{\bm{n}} \right)$. Then
\begin{align}
\label{eq:szego1rewritten2}
	& A_{\bm{n}}^{-1}\begin{pmatrix}
	\Phi_{\bm{n}}-z \Phi_{\bm{n}-\bm{e}_1} 
	\\
	\vdots
	\\
	\Phi_{\bm{n}}-z \Phi_{\bm{n}-\bm{e}_r} 
\end{pmatrix}
= 
\begin{pmatrix}
	\Phi^*_{\bm{n}-\bm{e}_1} 
	\\
	\vdots
	\\
	\Phi^*_{\bm{n}-\bm{e}_r} 
\end{pmatrix}. 
\end{align}
\end{itemize}
\end{cor}
\begin{rem}
    \textup{For $r=1$,  $A_{\bm{n}}$ becomes $(1-(1-|\alpha_n|^2))/\bar\alpha_n = \alpha_n$, so that~\eqref{eq:szego1rewritten1} reduces to~\eqref{eq:scalar szego 1}.}
\end{rem}
\begin{proof}
    For (i), just plug ~\eqref{eq:first recurrence} into~\eqref{eq:third recurrence}. To show (ii), first apply elementary row operations to show that $\det (I-R_{\bm{n}}) = 1-\sum_{j=1}^r \rho_{\bm{n},j} = \alpha_{\bm{n}}\beta_{\bm{n}}$, by~\eqref{eq:sumRhos}. Then apply the same approach to compute the classical adjoint of $I-R_{\bm{n}}$. This produces the formula $A_{\bm{n}}^{-1}=\tfrac{1}{\alpha_{\bm{n}}} \left((1-\sum_{l=1}^r \rho_{\bm{n},l})I+R_{\bm{n}} \right)$.
\end{proof}

\section{Recurrence Relations for Type I Polynomials}\label{ss:typeI}

If an index $\bm{n} \neq \bm{0}$ is normal then $\bm{\Lambda}_{\bm{n}} = (\Lambda_{\bm{n},1},\dots,\Lambda_{\bm{n},r})$ and $\bm{\Lambda}_{\bm{n}}^* = (\Lambda^*_{\bm{n},1},\dots,\Lambda^*_{\bm{n},r})$ will stand for the unique polynomial vectors with the normalizations described in Remark~\ref{rem:a-d} (c) and (d), respectively.  The following simple lemma will be used a few times. It is one example of the so-called biorthogonality property, which can be proved in the exact same way as for MOPRL~\cite[Thm 23.1.6]{Ismail}. 
\begin{lemma}
    Suppose multi-indices $\bm{n}$ and $\bm{n}+\bm{e}_k$ are normal. Then
    \begin{align}\label{eq:biorth}
        \jap{\Phi_{\bm{n}},\Lambda_{\bm{n}+\bm{e}_k,k}}_k = \sum_{m = 1}^{r}\jap{\Phi_{\bm{n}},\Lambda_{\bm{n}+\bm{e}_k,m}}_m & = 1.
    \end{align}
\end{lemma}
\begin{proof}
For  $m\ne k$ we have $\deg \Lambda_{\bm{n}+\bm{e}_k,m} \le n_m-1$, so that $\jap{\Phi_{\bm{n}},\Lambda_{\bm{n}+\bm{e}_k,m}}_m = 0$. This proves the first equality, and for the second equality, 
\m{\sum_{m = 1}^{r}\jap{\Phi_{\bm{n}},\Lambda_{\bm{n}+\bm{e}_k,m}}_m
= \overline{\sum_{m = 1}^{r}\jap{\Lambda_{\bm{n}+\bm{e}_k,m},\Phi_{\bm{n}}}_m} =
\overline{\sum_{m = 1}^{r}\jap{\Lambda_{\bm{n}+\bm{e}_k,m},z^{\abs{\bm{n}}}}_m}
= 1.}
\end{proof}
\begin{rem}
    \textup{Denote $\kappa_{\bm{n},k}$ to be the leading $z^{n_k-1}$-coefficient of $\Lambda_{\bm{n},k}$. Then 
    \eqref{eq:biorth} implies 
    \begin{equation}\label{eq:normCoeff}
        \bar\kappa_{\bm{n}+\bm{e}_k,k}=
        \frac{1}{\jap{\Phi_{\bm{n}}, z^{n_k}}_k}.
    \end{equation}
    Note that $\jap{\Phi_{\bm{n}}, z^{n_k}}_k \neq 0 \neq \kappa_{\bm{n}+\bm{e}_k,k}$, which can also be seen from Lemma \ref{lem:normality lemma}.}
\end{rem}

The following result is the type I analogue of Theorem \ref{thm:type 2 szego}. From these one can easily deduce the analogue of the remaining two Szeg\H{o} recurrences, similarly to Corollaries~\ref{cor:szTypeIIextra1} and~\ref{cor:szTypeIIextra2}. 



\begin{thm}\leavevmode
\begin{itemize}
\item[(i)]
    If $\bm{n}$ is normal, along with all neighbouring indices $\bm{n}\pm\bm{e}_1,\dots,\bm{n}\pm\bm{e}_r$ that belong to $\Z_+^r$, then    \nm{eq:type 1 second recurrence}{z\bm{{\Lambda}}_{\bm{n}} = -\bar\beta_{\bm{n}}\bm{\Lambda}^*_{\bm{n}} + \sum_{j = 1}^{r}\bar\rho_{\bm{n},j}\bm{\Lambda}_{\bm{n}+\bm{e}_j}.}
    \item[(ii)] If $\bm{n}$ and $\bm{n}+\bm{e}_k$ are normal, then
    \nm{eq:type 1 first recurrence}{\bm{\Lambda}^*_{\bm{n}} = \bm{\Lambda}^*_{\bm{n}+\bm{e}_k} -\bar\alpha_{\bm{n}}\bm{\Lambda}_{\bm{n}+\bm{e}_k}.}
\end{itemize}
\end{thm}


\begin{proof}
(i) Choose $\delta_{\bm{n}}$ such that
\m{\sum_{m = 1}^{r}\jap{z\Lambda_{\bm{n},m} - \delta_{\bm{n}}\Lambda^*_{\bm{n},m}, z^p}_m = 0,\qquad p = 0,1,\dots,\abs{\bm{n}}-1.}
The orthogonality relations above hold for every choice of $\delta_{\bm{n}}$ except in the case $p = 0$.
In other words, we make the choice
\m{\delta_{\bm{n}} = \sum_{m = 1}^{r}\jap{z\Lambda_{\bm{n},m},1}_m.}
Since $\Phi_{\bm{n}}^* = \beta_{\bm{n}}z^{\abs{\bm{n}}} + \dots + 1$ we get
\m{\delta_{\bm{n}} & = \sum_{m = 1}^{r}\jap{z\Lambda_{\bm{n},m},\Phi_{\bm{n}}^*}_m-\sum_{m = 1}^{r}\jap{z\Lambda_{\bm{n},m},\beta_{\bm{n}}z^{\abs{\bm{n}}}}_m \\ & = \sum_{m = 1}^{r}\overline{\jap{\Phi_{\bm{n}}^*,z\Lambda_{\bm{n},m}}}_m-\bar\beta_{\bm{n}}\sum_{m = 1}^{r}\jap{\Lambda_{\bm{n},m},z^{\abs{\bm{n}}-1}}_m  = -\bar\beta_{\bm{n}}.}
Consider the set of vectors of polynomials $(\Xi_1,\dots,\Xi_r)$ with $\deg{\Xi_j} \leq n_j$ and
\m{\sum_{m = 1}^{r}\jap{\Xi_m, z^p}_m = 0,\hspace{0.5cm} p = 0,1,\dots,\abs{\bm{n}}-1.}
Rewriting this as the system of equations for the coefficients of these polynomials results in a homogeneous linear system with coefficient matrix equal to $M_{\bm{n}}^T$ but with $r$ columns added. Hence the null space of this matrix has dimension $r$. Note that $(\Lambda_{\bm{n}+\bm{e}_j,1},\dots,\Lambda_{\bm{n}+\bm{e}_j,r})$ is a solution for each $j = 1,\dots,r$. Moreover, these vectors are linearly independent. To see this, suppose
\m{\sum_{j = 1}^{r}c_j(\Lambda_{\bm{n}+\bm{e}_j,1},\dots,\Lambda_{\bm{n}+\bm{e}_j,r}) = 0.}
By Lemma \ref{lem:normality lemma}, $\deg{\Lambda_{\bm{n}+\bm{e}_m,m}} = n_m$, but $\deg{\Lambda_{\bm{n}+\bm{e}_j,m}} = n_m - 1$ if $j \neq m$, so we must have $c_m = 0$. Hence the vectors of the form $(\Lambda_{\bm{n}+\bm{e}_j,1},\dots,\Lambda_{\bm{n}+\bm{e}_j,r})$ form a basis of the solution space of our linear system, and therefore
\m{z\bm{\Lambda}_{\bm{n}} + \bar\beta_{\bm{n}}\bm{\Lambda}^*_{\bm{n}} = \sum_{j = 1}^{r}\sigma_{\bm{n},j}\bm{\Lambda}_{\bm{n}+\bm{e}_j}}
for some complex numbers $\sigma_{\bm{n},1},\dots,\sigma_{\bm{n},r}$. By comparing the leading coefficients in the above recurrence relation, we see that
\nm{eq:kappa leading coefficient formula}{\kappa_{\bm{n},k} = \sigma_{\bm{n},k}\kappa_{\bm{n}+\bm{e}_k,k},}
where $\kappa_{\bm{n},j}$ was the  $z^{n_j-1}$-coefficient of $\Lambda_{\bm{n},j}$. On the other hand, by taking the $k$-th inner product with respect to $z^{n_k}$ on both sides of \eqref{eq:second recurrence} we see that
\nm{eq:rho inner product formula}{\jap{\Phi_{\bm{n}},z^{n_k}}_k = \rho_{\bm{n},k}\jap{\Phi_{\bm{n}-\bm{e}_k},z^{n_k-1}}_k.}
By combining these two relations we get
\m{\bar\sigma_{\bm{n},k}\bar\kappa_{\bm{n}+\bm{e}_k,k}\jap{\Phi_{\bm{n}},z^{n_k}}_k = \rho_{\bm{n},k}\bar\kappa_{\bm{n},k}\jap{\Phi_{\bm{n}-\bm{e}_k},z^{n_k-1}}_k.}
If $\bm{n}+\bm{e}_k$ is normal then the left hand side is equal to $\bar\sigma_{\bm{n},k}$, by \eqref{eq:normCoeff}. If $n_k = 0$ then the right hand side vanishes and then $\sigma_{\bm{n},k} = 0 = \bar\rho_{\bm{n},k}$, and if $\bm{n}-\bm{e}_k$ is normal then the right hand side is equal to $\rho_{\bm{n},k}$ by \eqref{eq:normCoeff}, so $\sigma_{\bm{n},k} = \bar\rho_{\bm{n},k}$.

\medskip

(ii) We have the orthogonality relations
\m{\sum_{m = 1}^{r}\jap{\Lambda^*_{\bm{n}+\bm{e}_k,m}-\Lambda^*_{\bm{n},m}, z^p}_m = 0,\hspace{0.5cm} p = 0,1,\dots,\abs{\bm{n}}-1.}
Since $\deg{(\Lambda^*_{\bm{n}+\bm{e}_k,m}-\Lambda^*_{\bm{n},m})} \leq n_m - 1$ if $m \neq k$ and $\deg{(\Lambda^*_{\bm{n}+\bm{e}_k,k}-\Lambda^*_{\bm{n},k})} \leq n_k$ there must be a constant $\epsilon_{\bm{n},k}$ such that
\m{\bm\Lambda^*_{\bm{n}+\bm{e}_k}-\bm\Lambda^*_{\bm{n}} = \epsilon_{\bm{n},k}\bm\Lambda_{\bm{n}+\bm{e}_k}.}
Now taking inner products with $z^{\abs{\bm{n}}}$ yields
\m{\epsilon_{\bm{n},k} = -\sum_{m = 1}^r \jap{\Lambda^*_{\bm{n},m}, z^{\abs{\bm{n}}}}_m.}
Since $\Phi_{\bm{n}} = z^{\abs{\bm{n}}} + \dots + \alpha_{\bm{n}}$ we get
\m{\epsilon_{\bm{n},k} & = -\brkt{\sum_{m = 1}^r \jap{\Lambda^*_{\bm{n},m}, \Phi_{\bm{n}}}_m - \sum_{m = 1}^r \jap{\Lambda^*_{\bm{n},m},\alpha_{\bm{n}}}_m} \\ & = -\sum_{m = 1}^r \overline{\jap{\Phi_{\bm{n}}, \Lambda^*_{\bm{n},m}}}_m + \bar\alpha_{\bm{n}}\sum_{m = 1}^{r}\jap{\Lambda^*_{\bm{n},m}, 1}_m = \bar\alpha_{\bm{n}}.}
\end{proof}

\section{Compatibility Conditions}\label{ss:compatibility}

Proposition \ref{pr:compPoly1} is identical to an important and peculiar feature of MOPRL, that does not appear in OPRL. However for MOPUC, this structure is richer with the presence of both $\Phi_{\bm{n}}$ and $\Phi^*_{\bm{n}}$, as Proposition \ref{pr:compPoly2} shows.

\begin{prop}\label{pr:compPoly1}
If $\bm{n}+\bm{e}_k$, $\bm{n}+\bm{e}_l$, and $\bm{n}$ are normal, and $k \neq l$, then there is a complex number $\gamma_{\bm{n}}^{kl}$ such that
\nm{eq:compPoly1}{\Phi_{\bm{n}+\bm{e}_k}-\Phi_{\bm{n}+\bm{e}_l} = \gamma_{\bm{n}}^{kl}\Phi_{\bm{n}}.}
\end{prop}
\begin{prop}\label{pr:compPoly2}
    Assume 
    that all the $\Z_+^r$ multi-indices that appear in the corresponding equations are 
    normal. Then 
    \begin{align}
        \label{eq:compPoly3} & 
        \Phi^*_{\bm{n} + \bm{e}_k} - \Phi^*_{\bm{n} + \bm{e}_l} 
        = \beta_{\bm{n} + \bm{e}_k + \bm{e}_l} z (\Phi_{\bm{n} + \bm{e}_l}-\Phi_{\bm{n} + \bm{e}_k} ),
		\\
  \label{eq:compPoly2}  &\Phi^*_{\bm{n} + \bm{e}_k} - \Phi^*_{\bm{n} +\bm{e}_l} =(\beta_{\bm{n}+\bm{e}_k}-\beta_{\bm{n}+\bm{e}_l} ) z \Phi_{\bm{n}}, \\
  \label{eq:compPoly4} &  \beta_{\bm{n}+\bm{e}_k}  (\Phi^*_{\bm{n} + \bm{e}_l} -  \Phi^*_{\bm{n}}) =  \beta_{\bm{n}+\bm{e}_l}  (\Phi^*_{\bm{n} + \bm{e}_k} -  \Phi^*_{\bm{n}} ).
    \end{align}
\end{prop}
\begin{proof}[Proof of Propositions~\ref{pr:compPoly1} and~\ref{pr:compPoly2}] 
Assuming $\bm{n}+\bm{e}_k$ and $\bm{n}+\bm{e}_l$ are normal, then $\Phi_{\bm{n}+\bm{e}_k}-\Phi_{\bm{n}+\bm{e}_l}$ has degree $\leq \abs{\bm{n}}$ and satisfies all the orthogonality relations at $\bm{n}$. This shows~\eqref{eq:compPoly1}. 


Assuming normality of $\bm{n}+\bm{e}_k$, $\bm{n}+\bm{e}_l$ and $\bm{n}+\bm{e}_k+\bm{e}_l$, ~\eqref{eq:first recurrence} gives
\m{& \Phi_{\bm{n} + \bm{e}_k + \bm{e}_l}^* = \Phi_{\bm{n}+\bm{e}_k}^* + \beta_{\bm{n}+\bm{e}_k+\bm{e}_l}z\Phi_{\bm{n}+\bm{e}_k}, \\ &
\Phi_{\bm{n} + \bm{e}_k + \bm{e}_l}^* = \Phi_{\bm{n}+\bm{e}_l}^* + \beta_{\bm{n}+\bm{e}_k+\bm{e}_l}z\Phi_{\bm{n}+\bm{e}_k}.}
Now combine these to get ~\eqref{eq:compPoly3}. On the other hand, if $\bm{n}+\bm{e}_k$, $\bm{n}+\bm{e}_l$ and $\bm{n}$ are normal, then ~\eqref{eq:first recurrence} gives
\m{& \Phi_{\bm{n} + \bm{e}_k}^* = \Phi_{\bm{n}}^* + \beta_{\bm{n}+\bm{e}_k}z\Phi_{\bm{n}}, \\ &
\Phi_{\bm{n} + \bm{e}_l}^* = \Phi_{\bm{n}}^* + \beta_{\bm{n}+\bm{e}_l}z\Phi_{\bm{n}},}
and by combining these we get~\eqref{eq:compPoly2}. If we first multiply by $\beta_{\bm{n}+\bm{e}_l}$ in the first equation and $\beta_{\bm{n}+\bm{e}_l}$ in the second equation we would instead get ~\eqref{eq:compPoly4}.
\end{proof}




We could easily derive analogue results for type I polynomials, using similar methods. We only prove the analogue of \eqref{eq:compPoly1} and leave the analogues of~\eqref{eq:compPoly3},~\eqref{eq:compPoly2},~\eqref{eq:compPoly4} as a quick exercise to the interested reader.

\begin{prop}
If all indices appearing below are normal, and $k \neq l$, then
\nm{eq:diagonal recurrence type 1}{\bm\Lambda_{\bm{n}-\bm{e}_k} - \bm\Lambda_{\bm{n}-\bm{e}_l} = \bar\gamma_{\bm{n}-\bm{e}_k-\bm{e}_l}^{kl}\bm\Lambda_{\bm{n}}.}
\end{prop}


\begin{proof}
Note that
\nm{eq:type 1 compatibility derivation}{\sum_{m = 1}^{r}\jap{\Lambda_{\bm{n}-\bm{e}_k,m} - \Lambda_{\bm{n}-\bm{e}_l,m}, z^p}_m = 0, \quad p = 0,1,\dots,\abs{\bm{n}}-2,}
and $\deg(\Lambda_{\bm{n}-\bm{e}_k,m} - \Lambda_{\bm{n}-\bm{e}_l,m}) \leq n_m - 1$, so we get 
\nm{eq:eta recurrence}{\bm\Lambda_{\bm{n}-\bm{e}_k} - \bm\Lambda_{\bm{n}-\bm{e}_l} = \eta_{\bm{n}}^{kl}\bm\Lambda_{\bm{n}}.}
for some constant $\eta_{\bm{n}}^{kl}$. If we now compare the degrees in \eqref{eq:eta recurrence} when $m = k$ we get
\m{-\kappa_{\bm{n}-\bm{e}_l,k} = \kappa_{\bm{n},k}\eta_{\bm{n}}^{kl}.}
On the other hand, if we take the $k$-th inner product with $z^{n_k}$ in \eqref{eq:compPoly1}, and shift the indices from $\bm{n}$ to $\bm{n}-\bm{e}_k-\bm{e}_l$ we get
\nm{eq:gamma inner product formula}{-\jap{\Phi_{\bm{n}-\bm{e}_k},z^{n_k-1}}_k = \jap{\Phi_{\bm{n}-\bm{e}_k-\bm{e}_l},z^{n_k-1}}_k\gamma_{\bm{n}-\bm{e}_k-\bm{e}_l}^{kl}.}
Combining these relations produces
\m{\bar\kappa_{\bm{n},k}\jap{\Phi_{\bm{n}-\bm{e}_k},z^{n_k-1}}_k\bar\eta_{\bm{n}}^{kl} = \bar\kappa_{\bm{n}-\bm{e}_l,k}\jap{\Phi_{\bm{n}-\bm{e}_k-\bm{e}_l},z^{n_k-1}}_k\gamma_{\bm{n}-\bm{e}_k-\bm{e}_l}^{kl}.}
By \eqref{eq:normCoeff} we then get $\eta_{\bm{n}}^{kl} = \bar\gamma_{\bm{n}-\bm{e}_k-\bm{e}_l}^{kl}$.
\end{proof}

In MOPRL, the nearest-neighbour recurrence coefficients satisfy a set of partial difference equations (see \cite{NNR}). The same methods that were used to prove this result can be applied to MOPUC, to get a similar set of equations. However. we choose to give a shorter proof using a different approach. 

\begin{thm}
    We have the compatibility conditions
    \begin{align} 
    \label{eq:compatibility condition 1}
    & {\beta_{\bm{n}}(\alpha_{\bm{n}-\bm{e}_l} - \alpha_{\bm{n}-\bm{e}_k}) = (\beta_{\bm{n}-\bm{e}_k} - \beta_{\bm{n}-\bm{e}_l})\alpha_{\bm{n}-\bm{e}_k-\bm{e}_l},}
    \\
    \label{eq:compatibility condition 2}
    & {\alpha_{\bm{n}}\beta_{\bm{n}} + \sum_{j = 1}^r \rho_{\bm{n},j} = 1,}
    \\
    \label{eq:compatibility condition 3.1}
    & {(\alpha_{\bm{n}-\bm{e}_l}-\alpha_{\bm{n}-\bm{e}_k})\alpha_{\bm{n}-\bm{e}_l}\rho_{\bm{n},k} = (\alpha_{\bm{n}+\bm{e}_k-\bm{e}_l}-\alpha_{\bm{n}})\alpha_{\bm{n}-\bm{e}_k-\bm{e}_l}\rho_{\bm{n}-\bm{e}_l,k},}
    \end{align}
    assuming that all the $\Z_+^r$ multi-indices that appear in the corresponding equations are normal, and for \eqref{eq:compatibility condition 2} we also need normality of all indices $\bm{n}-\bm{e}_1,\dots,\bm{n}-\bm{e}_r$ that belong to $\Z_+^r$.
\end{thm}


\begin{proof}
By putting $z = 0$ in ~\eqref{eq:compPoly1} we see that 
\nm{eq:gamma alpha}{\alpha_{\bm{n}+\bm{e}_k} - \alpha_{\bm{n}+\bm{e}_l} = \alpha_{\bm{n}}\gamma_{\bm{n}}^{kl}.}
Similarly, if we combine ~\eqref{eq:compPoly1} and ~\eqref{eq:compPoly3} and compare leading coefficients we see that 
\nm{eq:gamma beta}{\beta_{\bm{n}+\bm{e}_k+\bm{e}_l} \gamma_{\bm{n}}^{kl}=\beta_{\bm{n}+\bm{e}_l}-\beta_{\bm{n}+\bm{e}_k}.} 
If we now combine \eqref{eq:gamma alpha} and \eqref{eq:gamma beta} we get \eqref{eq:compatibility condition 1}, and if we put $z = 0$ in \eqref{eq:third recurrence} we get \eqref{eq:compatibility condition 2}, so what remains is to prove \eqref{eq:compatibility condition 3.1}. From \eqref{eq:gamma inner product formula} we get the inner product formulas
\m{\gamma_{\bm{n}-\bm{e}_k-\bm{e}_l}^{kl} = -\frac{\jap{\Phi_{\bm{n}-\bm{e}_k},z^{n_k-1}}_k }{\jap{\Phi_{\bm{n}-\bm{e}_k-\bm{e}_l},z^{n_k-1}}_k},\quad \gamma_{\bm{n}-\bm{e}_l}^{kl} = -\frac{\jap{\Phi_{\bm{n}},z^{n_k}}_k }{\jap{\Phi_{\bm{n}-\bm{e}_l},z^{n_k}}_k}.}
We can combine these two relations to get
\m{\jap{\Phi_{\bm{n}},z^{n_k}}_k\jap{\Phi_{\bm{n}-\bm{e}_k-\bm{e}_l},z^{n_k-1}}_k\gamma_{\bm{n}-\bm{e}_k-\bm{e}_l}^{kl} = \jap{\Phi_{\bm{n}-\bm{e}_k},z^{n_k-1}}_k\jap{\Phi_{\bm{n}-\bm{e}_l},z^{n_k}}_k\gamma_{\bm{n}-\bm{e}_l}^{kl}.}
Dividing by $\jap{\Phi_{\bm{n}-\bm{e}_k-\bm{e}_l},z^{n_k-1}}_k\jap{\Phi_{\bm{n}-\bm{e}_k},z^{n_k-1}}_k$ and using \eqref{eq:rho inner product formula} we arrive to
\nm{eq:compatibility condition 3.2}{\rho_{\bm{n},k}\gamma_{\bm{n}-\bm{e}_k-\bm{e}_l}^{kl} = \rho_{\bm{n}-\bm{e}_l,k}\gamma_{\bm{n}-\bm{e}_l}^{kl}.}
Now multiply by $\alpha_{\bm{n}-\bm{e}_l}\alpha_{\bm{n}-\bm{e}_k - \bm{e}_l}$, and use \eqref{eq:gamma alpha} to get \eqref{eq:compatibility condition 3.1}. 
\end{proof}

\begin{rem} \textup{As we saw in the proof, \eqref{eq:compatibility condition 3.1} has the alternative version \eqref{eq:compatibility condition 3.2}.
If we multiply by $\beta_{\bm{n}}\beta_{\bm{n}+\bm{e}_k}$ and use \eqref{eq:gamma beta} we get another version, of the form
\nm{eq:compatibility condition 3}{(\beta_{\bm{n}-\bm{e}_k}-\beta_{\bm{n}-\bm{e}_l})\beta_{\bm{n}+\bm{e}_k}\rho_{\bm{n},k} = (\beta_{\bm{n}}-\beta_{\bm{n}+\bm{e}_k-\bm{e}_l})\beta_{\bm{n}}\rho_{\bm{n}-\bm{e}_l,k}.}}
\end{rem}

\section{Christoffel--Darboux Formula}\label{ss:CD}
\begin{thm}
    Let $(\bm{n}_k)_{k = 0}^{N}$ be a path of multi-indices such that $\bm{n}_0 = \bm{0}$, 
    and $\bm{n}_{k+1}-\bm{n}_k = \bm{e}_{l_k}$ for some $1\le l_k\le r$, in particular $|\bm{n}_k|=k$. Assume all multi-indices on the path are normal, along with all the neighbouring indices that belong to $\Z_+^r$. Then we have the Christoffel--Darboux formula
    \nm{eq:cd formula}{(1 - z\bar\zeta)\sum_{k = 0}^{N-1} \Phi_{\bm{n}_k}(z)\overline{\bm\Lambda_{\bm{n}_{k+1}}(\zeta)} = \Phi^*_{\bm{n}_N}(z)\overline{\bm\Lambda^*_{\bm{n}_N}(\zeta)} - \sum_{j = 1}^{r}\rho_{\bm{n}_N,j}z\Phi_{\bm{n}_N-\bm{e}_j}(z)\overline{\bm\Lambda_{\bm{n}_N+\bm{e}_j}(\zeta)} .}
\end{thm}
\begin{proof}
    By \eqref{eq:type 1 second recurrence} and \eqref{eq:first recurrence} we have
    \m{z\bar\zeta\Phi_{\bm{n}_k}(z)\overline{\bm\Lambda_{\bm{n}_{k+1}}(\zeta)} = -\beta_{\bm{n}_{k+1}}z\Phi_{\bm{n}_k}(z)\overline{\bm\Lambda^*_{\bm{n}_{k+1}}(\zeta)} + \sum_{j = 1}^r \rho_{\bm{n}_{k+1},j}z\Phi_{\bm{n}_k}(z)\overline{\bm\Lambda_{\bm{n}_{k+1}+\bm{e}_j}(\zeta)} \\ = \Phi^*_{\bm{n}_k}(z)\overline{\bm\Lambda^*_{\bm{n}_{k+1}}(\zeta)} - \Phi^*_{\bm{n}_{k+1}}(z)\overline{\bm\Lambda^*_{\bm{n}_{k+1}}(\zeta)} + \sum_{j = 1}^r \rho_{\bm{n}_{k+1},j}z\Phi_{\bm{n}_k}(z)\overline{\bm\Lambda_{\bm{n}_{k+1}+\bm{e}_j}(\zeta)}.}
    By \eqref{eq:second recurrence} and \eqref{eq:type 1 first recurrence} we also have
    \m{\Phi_{\bm{n}_k}(z)\overline{\bm\Lambda_{\bm{n}_{k+1}}(\zeta)} = \alpha_{\bm{n}_k}\Phi^*_{\bm{n}_k}(z)\overline{\bm\Lambda_{\bm{n}_{k+1}}(\zeta)} + \sum_{j = 1}^{r}\rho_{\bm{n}_k,j}z\Phi_{\bm{n}_k-\bm{e}_j}(z)\overline{\bm\Lambda_{\bm{n}_{k+1}}(\zeta)} \\ = \Phi^*_{\bm{n}_k}(z)\overline{\bm\Lambda^*_{\bm{n}_{k+1}}(\zeta)} - \Phi^*_{\bm{n}_k}(z)\overline{\bm\Lambda^*_{\bm{n}_{k}}(\zeta)} + \sum_{j = 1}^{r}\rho_{\bm{n}_k,j}z\Phi_{\bm{n}_k-\bm{e}_j}(z)\overline{\bm\Lambda_{\bm{n}_{k+1}}(\zeta)}.}
    Putting these together yields
    \m{(1 - z\bar\zeta)\Phi_{\bm{n}_k}(z)\overline{\bm\Lambda_{\bm{n}_{k+1}}(\zeta)} & = \Phi^*_{\bm{n}_{k+1}}(z)\overline{\bm\Lambda^*_{\bm{n}_{k+1}}(\zeta)} - \Phi^*_{\bm{n}_k}(z)\overline{\bm\Lambda^*_{\bm{n}_{k}}(\zeta)} \\ & + \sum_{j = 1}^{r}\rho_{\bm{n}_k,j}z\Phi_{\bm{n}_k-\bm{e}_j}(z)\overline{\bm\Lambda_{\bm{n}_{k+1}}(\zeta)} \\ & - \sum_{j = 1}^r \rho_{\bm{n}_{k+1},j}z\Phi_{\bm{n}_k}(z)\overline{\bm\Lambda_{\bm{n}_{k+1}+\bm{e}_j}(\zeta)}.}
    By \eqref{eq:compPoly1} and \eqref{eq:diagonal recurrence type 1} 
    we can write
    \begin{align*}
    \sum_{j = 1}^{r}\rho_{\bm{n}_k,j} z
     \Phi_{\bm{n}_k-\bm{e}_j}&(z)
      \overline{\bm\Lambda_{\bm{n}_{k+1}}(\zeta)} - \sum_{j = 1}^r \rho_{\bm{n}_{k+1},j}z\Phi_{\bm{n}_k}(z)\overline{\bm\Lambda_{\bm{n}_{k+1}+\bm{e}_j}(\zeta)} 
    \\ 
     = \sum_{j = 1}^{r}
      \rho_{\bm{n}_k,j}
      z
      &\Phi_{\bm{n}_k-\bm{e}_j}(z)
      \brkt{\overline{\bm\Lambda_{\bm{n}_{k}+\bm{e}_j}(\zeta)} + \gamma_{\bm{n}_{k}}^{jl_k}\overline{\bm\Lambda_{\bm{n}_{k+1}+\bm{e}_j}(\zeta)}}
    \\ 
     - \sum_{j = 1}^r &\rho_{\bm{n}_{k+1},j}
     z\brkt{\Phi_{\bm{n}_{k+1}-\bm{e}_j}(z) - \gamma_{\bm{n}_{k}-\bm{e}_j}^{l_kj} \Phi_{\bm{n}_k-\bm{e}_j}(z)}\overline{\bm\Lambda_{\bm{n}_{k+1}+\bm{e}_j}(\zeta)} 
     \\
    = \sum_{j =1}^{r} \big(\rho_{\bm{n}_k,j}  & 
     \gamma_{\bm{n}_{k}}^{jl_k} - \rho_{\bm{n}_{k+1},j}\gamma_{\bm{n}_{k}-\bm{e}_j}^{jl_k}\big)
     z\Phi_{\bm{n}_k-\bm{e}_j}(z)\overline{\bm\Lambda_{\bm{n}_{k+1}+\bm{e}_j}(\zeta)} 
    \\ 
    + \sum_{j = 1}^{r} &\rho_{\bm{n}_k,j}z 
    \Phi_{\bm{n}_k-\bm{e}_j}(z)\overline{\bm\Lambda_{\bm{n}_{k}+\bm{e}_j}(\zeta)}  - \sum_{j = 1}^{r}\rho_{\bm{n}_{k+1},j}z\Phi_{\bm{n}_{k+1}-\bm{e}_j}(z)\overline{\bm\Lambda_{\bm{n}_{k+1}+\bm{e}_j}(\zeta)}.
\end{align*}
\noindent Since $\rho_{\bm{n}_k,j}\gamma_{\bm{n}_{k}}^{jl_k} - \rho_{\bm{n}_{k+1},j}\gamma_{\bm{n}_{k}-\bm{e}_j}^{jl_k}=0$
    (see \eqref{eq:compatibility condition 3.2}), we end up with
    \m{(1 - z\bar\zeta)\Phi_{\bm{n}_k}(z)\overline{\bm\Lambda_{\bm{n}_{k+1}}(\zeta)} & = \Phi^*_{\bm{n}_{k+1}}(z)\overline{\bm\Lambda^*_{\bm{n}_{k+1}}(\zeta)} - \Phi^*_{\bm{n}_k}(z)\overline{\bm\Lambda^*_{\bm{n}_{k}}(\zeta)} \\ & + \sum_{j = 1}^{r}\rho_{\bm{n}_k,j}z\Phi_{\bm{n}_k-\bm{e}_j}(z)\overline{\bm\Lambda_{\bm{n}_{k}+\bm{e}_j}(\zeta)} \\ & - \sum_{j = 1}^{r}\rho_{\bm{n}_{k+1},j}z\Phi_{\bm{n}_{k+1}-\bm{e}_j}(z)\overline{\bm\Lambda_{\bm{n}_{k+1}+\bm{e}_j}(\zeta)}.}
    Now summation over $k$ leads to a telescoping sum resulting in exactly \eqref{eq:cd formula}. 
\end{proof}

\bibsection

\end{document}